\documentclass[smallextended,
final
]{svjour3}%
\usepackage{graphicx} %
\usepackage{mathptmx}%
\usepackage{times}%
\usepackage{amsmath, amssymb, xcolor} %

\usepackage[final]{listings}
\def\cdot{}
\smartqed
\let\tilde=\widetilde

\long\def\beginpgfgraphicnamed#1#2\endpgfgraphicnamed{\includegraphics{#1}}
\newlength{\figurewidth}
\DeclareMathOperator{\co}{co}
\newcommand{\field}[1]{\ensuremath{\mathbb{#1}}}
\newcommand{\R}{\field{R}}

\newcommand{\Rp}{{\field{R}_+}}
\newcommand{\Rnp}{\ensuremath{\field{R}^n_+}}
\newcommand{\Rn}{\ensuremath{\field{R}^n}}

\newcommand{\Rpnn}{\ensuremath{\field{R}^{n\times n}_+}}

\newcommand{\K}{\ensuremath{\mathcal{K}}}

\newcommand{\Kinf}{\ensuremath{\mathcal{K}_\infty}}
\newcommand{\id}{\ensuremath{\mbox{id}}}
\DeclareMathOperator{\diag}{diag}
\DeclareMathOperator{\Int}{int}

\let\epsilon=\varepsilon

\usepackage{algorithmic}
\usepackage{algorithm}
\usepackage[numbers,sort,compress]{natbib}

\title{%
Stability verification for monotone systems using homotopy algorithms
}%
\author{Bj\"orn S.~R\"uffer \and Fabian R.\ Wirth}%
\institute{%
  B.S.\ R\"uffer \at %
  Department of Electrical and Electronic Engineering, %
  University of Melbourne, %
  Parkville VIC 3010, %
  Australia %
  \email{bjoern@rueffer.info} %
  \and %
  F.R.\ Wirth \at %
  Institut f\"ur Mathematik, %
  Universit\"at W\"urzburg, %
  Am Hubland, %
  97074  W\"urzburg, %
  Germany %
  \email{wirth@mathematik.uni-wuerzburg.de} %
}%

\date{Received: date / Accepted: date}%
\begin{document}
\maketitle

\begin{abstract}
  A  monotone  self-mapping  of  the  nonnegative  orthant  induces  a
monotone  discrete-time dynamical  system  which evolves  on the  same
orthant. If with respect to  this system the origin is attractive then
there  must  exists points  whose  image  under  the monotone  map  is
strictly  smaller  than  the  original point,  in  the  component-wise
partial  ordering. Here  it  is shown  how  such points  can be  found
numerically, leading to  a recipe to compute order  intervals that are
contained in the region of  attraction and where the monotone map acts
essentially  as  a  contraction.   An  important  application  is  the
numerical verification of  so-called generalized small-gain conditions
that appear in the stability theory of large-scale systems.
\end{abstract}

\section{Introduction}
\label{sec:introduction}

By $\Rp$ we denote the nonnegative real numbers, $[0,\infty)$. A class
$\K$ function is a continuous function $\gamma\colon\Rp\to\Rp$ that
satisfies $\gamma(0)=0$ and is strictly increasing.  The function
$\gamma$ is of class $\Kinf\subset\K$ if in addition $\gamma$ is
unbounded. Note that with respect to composition the class $\Kinf$ is
a group and the class $\K$ a semi-group. Moreover, sums and positive multiples
of $\K$ functions are again $\K$ functions.

The nonnegative orthant $\Rnp$ induces a partial order on $\Rn$, which
coincides with the component wise ordering, and we write for
$x,y\in\Rn$, $x\leq y$ if $y-x\in\Rnp$, $x<y$ if [$x\leq y$ and $x\ne
y$], and $x\ll y$ if $y-x\in\Int\Rnp$, the interior of $\Rnp$.

A map $T\colon \Rnp\to\Rnp$ is monotone if $x\leq y$ implies $Tx\leq
Ty$. For any $\Kinf$ function $\rho$, the map
$D=\diag(\rho)\colon\Rnp\to\Rnp$ defined by $(Dx)_{i}=\rho(x_{i})$ is
an example of a monotone map.

We consider the following problem:
\begin{problem}
  \label{problem:1}
  Let a monotone, continuous map $T\colon \Rnp\to\Rnp$ with $T(0)=0$
  and a real number $r>0$ be given. Find $s^{*}\in\Rnp$ satisfying
  \begin{enumerate}
  \item $T s^{*}\ll s^{*}$ (which implies $s^{*}\gg0$),
  \item $\|s^{*}\|_{1}=r$, and
  \item $T s \ngeq s$ for all $s\in [0,s^{*}], s\ne0$.\qed
  \end{enumerate}
\end{problem}
The existence of such an $s^{*}$ (for sufficiently small $r>0$) is a
necessary and sufficient condition for asymptotic stability of the
origin with respect to the discrete time system
\begin{equation}
  \label{eq:8}
  s^{+}=Ts,\quad s\in\Rnp.
\end{equation}
It also arises as a so-called generalized small-gain condition
\cite{RufferDashkovskiy:2009:Local-ISS-of-large-scale-interconnection:}.
If it is satisfied then the set $[0,s^{*}]$ is contained in the region of
attraction. Moreover, in this case also the system with input,
\begin{equation}
  \label{eq:9}
  s^{+}=Ts+w,\quad s,w\in\Rnp,
\end{equation}
is \emph{locally input-to-state stable}
\cite{GaoLin:2000:On-equivalent-notions-of-input-to-state-:,
  JiangWang:2001:Input-to-state-stability-for-discrete-ti:}, and
knowledge of $s^{*}$ yields estimates for the sets of admissible
inputs and initial conditions which result in bounded outputs, cf.\
\cite{RufferDashkovskiy:2009:Local-ISS-of-large-scale-interconnection:}.

The numerical solution of this problem is interesting in two aspects:
First of all, it provides a numerical way to make a qualitative
assertion. Secondly, this assertion is not only of qualitative nature
(i.e., a system is stable in some sense), but also quantitative in
that an estimate for the region of attraction is obtained as well.

The numerical solution of this problem is interesting for several
applications: In the context of large-scale interconnections of
nonlinear systems the solution to Problem~\ref{problem:1} can assert
(local) input-to-state stability of interconnections of many systems
in arbitrary interconnection topology. This in turn can be useful for
formation
control~\cite{TannerPappasKumar:2004:Leader-to-formation-stability:}
or the effective implementation of decentralized model predictive
control~\cite{RaimondoMagniScattolini:2007:Decentralized-MPC-of-nonlinear-systems:-:}.
Furthermore, in the same context the knowledge of $s^{*}$ can be used
to find a locally Lipschitz continuous Lyapunov function for the
composite large-scale system.  Another application is in queuing
theory. Here the solution to Problem~\ref{problem:1} can be used to
ascertain that a given switching policy stabilizes a flow switching
network via the use of a monotone monodromy operator, cf.\
\cite{FeoktistovaMatveev:2009:Dynamic-interactive-stabilization-of-the:},
at least for specified range of initial buffer levels and bounded
inflow.

It is known that a solution to Problem~\ref{problem:1} must exist for
any $r>0$ if the origin is globally attractive with respect to
\eqref{eq:8}. In this case necessarily it holds that $Ts\ngeq s$ for
all $s\in\Rnp, s\ne0$, and by virtue of a topological fixed point
result, for every $r>0$, there exists an $s^{*}$ with
$\|s^{*}\|_{1}=r$ satisfying $Ts^{*}\ll s^{*}$. However, even if
$T\colon \Rnp\to\Rnp$ is monotone, continuous, and satisfies $T(0)=0$
as well as $Ts\ngeq s$ for all $s\in\Rnp, s\ne0$, the origin is not
necessarily \emph{globally} attractive with respect to~\eqref{eq:8}
(but it is so locally). For $T$ of particular form a sufficient
condition for global asymptotic stability of the origin with respect
to \eqref{eq:8} is the existence of a diagonal map
$D=\diag(\id+\rho)\colon \Rnp\to\Rnp$ with $\rho\in\Kinf$ such that
for all $s\ne0$, $(D\circ T)s\ngeq s$, or, equivalently $(T\circ D)
s\ngeq s$. In this case the system \eqref{eq:9} is input-to-state
stable \cite{Ruffer:2009:Monotone-inequalities-dynamical-systems-:,
  Ruffer:2009:Small-gain-conditions-and-the-comparison:}. 

Also known is that if $T$ satisfies $T(s\oplus v) = Ts \oplus Tv$ for
all $s,v\in\Rnp$, where $\oplus$ denotes component-wise maximization,
then $T$ must be of the form $(Ts)_{i}=\max_{j}\gamma_{ij}(s_{j})$ for
all $i,j$, where $\gamma_{ij}\colon\Rp\to\Rp$ are nondecreasing
functions. In this case $T$ has been termed \emph{max-preserving}. Here, the
condition $Ts\ngeq s$ for all $s\ne0$ is equivalent to the \emph{cycle
  condition}, which assumes that $\gamma_{i_{1}i_{2}} \circ
\gamma_{i_{2}i_{3}} \circ \ldots \circ \gamma_{i_{k-1}i_{k}} \circ
\gamma_{i_{k}i_{k}}<\id$ for all finite ordered sequences
$(i_{1},i_{2},\ldots,i_{k})\subset \{1,\ldots,n\}^{k}$. If this
condition holds then with $e$ denoting the vector $(1,\ldots,1)^{T}$
and $t>0$ one has for $q(t):=\max\{t\cdot e, T(t\cdot e), \ldots,
T^{n-1} (t\cdot e)\}\in\Rnp$ a continuous path $q\colon\Rp\to\Rnp$
which is unbounded and nondecreasing in every component and satisfies
$T(q(t))\leq q(t)$. Re-parametrization yields a path $\tilde q(r)$
satisfying $\|\tilde q(r)\|_{1}=r$, which can be interpreted as a
parametrized ``almost'' solution to Problem~\ref{problem:1}, cf.\
\cite{KarafyllisJiang:2009:A-Vector-Small-Gain-Theorem-for-General-:}.

In the linear case the action of $T$ can be represented by
multiplication with a nonnegative matrix (i.e., every component is
nonnegative), which we also denote by $T\in\R_{+}^{n\times n}$. It is
known that the following are equivalent: 
\begin{enumerate} 
\item $Ts\ngeq s$ for all $s\in\Rnp, s>0$;
\item there exists a $D=\diag(\id+\rho)$ with $\rho\in\Kinf$ such that
  $(T\circ D)s\ngeq s$ for all $s\in\Rnp, s>0$ (equivalently, $(D\circ
  T)s\ngeq s$ for all $s\in\Rnp, s>0$);
\item the spectral radius of $T$ is less than one;
\item there exists a unit vector $s^{*}\gg0$ (with respect to the
  1-norm) such that $Ts^{*}\ll s^{*}$, hence also the ray given by
  $r\cdot s^{*}$, $r\in\Rp$ satisfies $T(rs^{*})\ll rs^{*}$ for $r>0$ and
  $\|rs^{*}\|_{1}=r$;
\item the inverse of $I-T$ exists and is given by the nonnegative
  matrix $\sum_{k\geq0} T^{k}$.
\end{enumerate}
The existence of the vector $s^{*}$ is of course related to the
classical Perron-Frobenius theory. If $T$ is primitive then $s^{*}$ is
just the positive Perron-Frobenius root corresponding to the maximal
eigenvalue, which coincides with the spectral radius.  Further
extensions of the classical Perron-Frobenius theory exist for special
classes of nonlinear maps, in particular for homogeneous maps and for
concave maps, cf.\
\cite{AeyelsDe-Leenheer:2002:Extension-of-the-Perron-Frobenius-theore:,
  Krause:2001:Concave-Perron-Frobenius-theory-and-appl:}.

In this paper we propose the use of a homotopy method to find a point
$s^{*}$ satisfying $Ts^{*}\ll s^{*}$, $\|s^{*}\|_{1}$ for any given
$r>0$. Our method of choice is the K1 algorithm proposed by Eaves
\cite{Eaves:1972:Homotopies-for-computation-of-fixed-poin:} as a
computational version of a topological fixed point theorem. There are,
however, more elaborate choices of related algorithms available as
well, cf.\
\cite{AllgowerGeorg:1980:Simplicial-and-continuation-methods-for-:,
  AllgowerGeorg:1990:Numerical-continuation-methods:} for an
overview. While these algorithm have the potential of admitting faster
convergence, they tend to be more complicated to implement. One
particular advantage of homotopy methods is that they offer global
convergence: If there exists a point $s^{*}$ with $\|s^{*}\|_{1}=r$ with
the desired properties then it will be found. This has to bee seen in
contrast to methods based on Newton steps, which only guarantee
convergence if the algorithm is started sufficiently close to
$s^{*}$. Moreover, Newton methods usually assume some level of
smoothness, whereas homotopy methods only require continuity.

Once the point $s^{*}$ has been computed, the remaining verification
of property~3 in Problem~\ref{problem:1} is an easy task: It only
needs to be checked that the sequence $T^{k}(s^{*})$ converges to
zero. It is then a consequence of monotonicity that property~3 must
hold.

The paper is organized as follows.  In
Section~\ref{sec:monot-maps-monot} a few facts about monotone
self-mappings of the nonnegative orthant and their induced
discrete-time systems are recalled. In particular, the topological
fixed point theorem by Knaster, Kuratowski, and Mazurkiewicz is
discussed. A brief and informal description of some of the underlying
principles of Eaves' and other homotopy algorithms is given in
Section~\ref{sec:fixed-point-algor}. A MATLAB version of Eaves' K1
algorithm is provided in
Section~\ref{sec:Impl-k1}. Section~\ref{sec:algor-solut-probl}
explains a short procedure to solve Problem~\ref{problem:1} based on
the use of a homotopy algorithm and the computation of one trajectory
of system~\eqref{eq:8}. This is followed by several numerical examples
in Section~\ref{sec:examples}.

\section{Monotone maps and monotone discrete-time systems}
\label{sec:monot-maps-monot}

This section collects a few theoretical results from the literature. 

Throughout this section let $T\colon\Rnp\to\Rnp$ be monotone and
continuous with $T(0)=0$. Consider also the induced discrete-time
systems~\eqref{eq:8} and~\eqref{eq:9}. Denote their respective
solutions for initial condition $s^{0}\in\Rnp$ and, in case of
\eqref{eq:9}, input sequence $w=\{w(k)\in\Rnp\colon k\geq0\}$, at time
$k\geq0$ by $\phi_{\eqref{eq:8}}(k,s^{0})$ and
$\phi_{\eqref{eq:9}}(k,s^{0},w)$, respectively. If the reference to a
particular system is clear from the context we omit the reference to
the system. Observe that both systems satisfy the \emph{ordering of
  solutions principle:} If $s^{0}\leq v^{0}$ and for all $k\geq0$,
$w(k)\leq u(k)$ (which we abbreviate by $w\leq u$), then also
$\phi_{\eqref{eq:8}}(k,s^{0})\leq\phi_{\eqref{eq:8}}(k,v^{0})$ and
$\phi_{\eqref{eq:9}}(k,s^{0},w)\leq\phi_{\eqref{eq:9}}(k,v^{0},u)$ for
all $k\geq0$.

We denote the sphere in $\Rnp$ of radius $r>0$ with respect to the
1-norm by $S_{r}=\{s\in\Rnp\colon
\|s\|_{1}=\sum_{i}s_{i}=r\}$. Observe that $S_{r}$ is an $(n-1)$
simplex. 

\begin{theorem}
  \label{thm:1}
  Let $T\colon\Rnp\to\Rnp$ be monotone and continuous with $T(0)=0$.
  Assume that the origin is attractive with respect to~\eqref{eq:8}
  and denote the domain of attraction by $\mathcal{B}$. Then the
  following assertions holds:
  \begin{enumerate}
  \item For every $s\in\mathcal{B}$, $s\ne0$, necessarily $Ts\ngeq s$.
  \item If $R\in\R$, $R>0$ is such that $S_{R}\subset \mathcal{B}$
    then for all $r\in(0,R]$ there exists a point $s\in S_{r}$, $s\gg
    0$, satisfying $Ts\ll s$.
  \item The origin is stable in the sense of Lyapunov, i.e., for every
    $\epsilon>0$ there exists a $\delta>0$ such that
    $\|s\|_{1}<\delta$ implies $\|Ts\|_{1}<\epsilon$.
  \end{enumerate}
\end{theorem}
A proof can be found in
\cite{Ruffer:2009:Monotone-inequalities-dynamical-systems-:}. The
first assertion is not difficult to prove directly, and the third
follows from the second. The second assertion is the most technical,
and it is based on the covering theorem by
Knaster, Kuratowski, and Mazurkiewicz
(KKM)
\cite{KnasterKuratowskiMazurkiewicz:1929:Ein-Beweis-des-Fixpunktsatzes-fur-n-dime:}. Our
reasoning is based on the extension given in
\cite{Lassonde:1990:Sur-le-principe-KKM:} which allows to consider
coverings of open instead of closed sets.  The argument is basically
the following: One has $Ts\ngeq s$ for all $s\in\ S_{r}$, and $S_{r}$
is a simplex. This ordering condition implies that the simplex is
covered by the sets
\[ 
\Omega_{i}=\{s\in S_{r}\colon (Ts)_{i}< s_{i}\}.
\]
Moreover, no $\Omega_{i}$ can contain the face $\sigma_{i}\subset
S_{r}$ opposite of the vertex $r\cdot e_{i}\in \Omega_{i}$, where
$e_{i}$ denotes the $i$th unit vector. On the other hand, every
$k$-dimensional simplex $\sigma=\co\{r\cdot
e_{i}\}_{i\in\{i_{1},\ldots,i_{k}\}}$, $k\leq n-1$, is contained in
the union $\bigcup_{i\in\{i_{1},\ldots,i_{k}\}}\Omega_{i}$.
These are the prerequisites of the KKM Theorem, which then ascertains
that the intersection $\bigcap\Omega_{i}$ must be nonempty. The proof
of the KKM Theorem is based on the fact that every simplicial
refinement of $S_{r}$ must contain at least one special simplex, which
is again covered by all sets $\Omega_{i}$ but not by any strict
subclass of $\{\Omega_{1},\ldots,\Omega_{n}\}$. As the size of the
refinements tends to zero, this special simplex contracts to a point
--- the point of interest $s^{*}\in S_{r}$ satisfying $Ts^{*}\ll
s^{*}$.  In essence this proof technique relies on a fine
discretization of $S_{r}$ and an exhaustive search, which is not
implementation friendly if $n$ becomes large.

\section{Homotopy based fixed point algorithms}
\label{sec:fixed-point-algor}

An alternative proof of the KKM result has been given in
\cite{Eaves:1972:Homotopies-for-computation-of-fixed-poin:} by means
of a homotopy algorithm. In contrast to the original proof, here the
initial simplex is successively refined in every iteration, thus the
area that must contain $s^{*}$ becomes smaller and smaller as
iterations progress. We give a simplified account on the ideas
behind this algorithm.

As in the original KKM paper, each point in $s\in S_{r}$ is assigned
an integer label
\begin{equation}
  l(s)=\max \{i\colon s\in \Omega_{i}\}.\label{eq:11}
\end{equation}
Observe that $l(r\cdot e_{i})=i$ for $i=1,\ldots,n$. To keep things
uncluttered, let us call a set of $k$ distinct points (vertices)
$\sigma=\{v^{i}\in\Rnp\}_{i=1}^{k}$ a $(k-1)$-set.  The barycentric
centre of such a $(k-1)$-set is the point $v=\frac{1}{k}
\sum_{i=1}^{k}v_{k}$.  We call an $(n-1)$-set $\sigma$ \emph{complete}
if all vertices have distinct labels. Equivalently, every label
$1,\ldots,n$ gets assigned to a vertex exactly once. Observe that the
$(n-1)$-set $\sigma^{0}=\{r\cdot e_{i}\}_{i=1}^{n}$ is complete if
$Ts\ngeq s$ for all $s\in S_{r}$.

Now one can also consider an $n$-set $\tau$ obtained from a given
$(n-1)$-set $\sigma$ by adjoining one additional vertex $v$ taken from
the convex hull $\co(\sigma)$ of the vertices of $\sigma$. For example,
one could augment $\sigma$ with its barycentric centre to obtain such
an $n$-set $\tau$.

The following observation is at the core of the fixed point algorithms
by Eaves and also at the core of all related homotopy algorithms, cf.\
also~\cite[Thm.~1.7]{AllgowerGeorg:1980:Simplicial-and-continuation-methods-for-:}.

\begin{theorem}
  \label{thm:2}
  Every $n$-set has either none or exactly two complete
  $(n-1)$-subsets.
\end{theorem}

\begin{proof}
  Let $\tau=\{v^{1},\ldots, v^{n+1}\}$ denote the given $n$-set with
  its $n+1$ vertices.  Now either $l\big(\{v^{1},\ldots, v^{n+1}\}\big)$
  contains the set $\{1,\ldots,n\}$. In this case one label must get
  assigned twice, i.e., there exists a unique $i$ and $j,k$ such that
  $l(v^{j})=l(v^{k})=i$. In this case $\tau\setminus\{v^{j}\}$ and
  $\tau\setminus v^{k}$ are both complete $(n-1)$-sets. All other
  $(n-1)$-subsets must contain $v^{j}$ and $v^{k}$ and hence cannot be
  complete.  And if $\{1,\ldots,n\}\not\subseteq l\big(\{v^{1},\ldots,
  v^{n+1}\}\big)$ then no such $(n-1)$-set can exist.\qed
\end{proof}

The basic algorithm is now the following: $\sigma^{0}=\{r\cdot
e_{i}\}_{i=1}^{n}$ serves as a complete \emph{entry
  $(n-1)$-set}. Successively, another vertex is added, say the
barycentric centre $v$ of $\sigma^{0}$, to obtain the $n$-set
$\tau=\sigma^{0}\cup \{v\}$. By Theorem~\ref{thm:2}, $\tau$ contains
exactly one $(n-1)$-subset distinct from $\sigma^{0}$, which we denote
by $\sigma^{1}$. Progressing inductively, one obtains a sequence of
complete $(n-1)$-sets $\sigma^{k}$ whose area tends to zero.

The catch, however, is that the sequence does not necessarily contract
to a point. Instead, $\sigma^{k}$ may become ``long and thin'' as $k$
becomes large. To prevent this kind of behaviour, a more sophisticated
choice of new vertices $v$ is required. For this choice there exist a
variety of alternatives. The paper by
Eaves
\cite{Eaves:1972:Homotopies-for-computation-of-fixed-poin:} proposes
two such choices, which are named K1 and K2. Both guarantee
convergence. Of these two choices K1 is the easiest and shortest to
implement, and therefore we have chosen K1 for this exposition. On the
other hand, it does not converge as quickly as K2, as has already been
observed in
\cite{Eaves:1972:Homotopies-for-computation-of-fixed-poin:}. It should
be noted, however, that even more sophisticated pivoting strategies
and restart algorithms can be found in
\cite{AllgowerGeorg:1980:Simplicial-and-continuation-methods-for-:,
  AllgowerGeorg:1990:Numerical-continuation-methods:,
  AllgowerGeorg:1997:Numerical-path-following:} and the references
contained therein. A detailed discussion of these is far beyond the
scope of this paper. In principle any of these could have been used
instead of our particular choice for Eaves' K1 algorithm here.

\section{Implementation of Eaves' K1 algorithm}
\label{sec:Impl-k1}

The integer labeling function~\eqref{eq:11} is numerically not
feasible, instead we use the function
\begin{equation}
  \label{eq:12}
  l_{\epsilon}(s)=\max \{i\colon (Ts)_{i}+\epsilon\leq s_{i}\},
\end{equation}
where $\epsilon>0$ is a design parameter. It guarantees that if a
point $s^{*}$ is obtained with the fixed point algorithm, then this
point does in fact satisfy $Ts^{*}\ll s^{*}$ and not just $Ts^{*}<
s^{*}$. 

Obviously, if a point $s^{*}$ with $Ts^{*}\ll s^{*}$ exists at all,
then it has to be found by the algorithm if only $\epsilon>0$ is small
enough. On the other hand, from practice it is fair to say that larger
$\epsilon$  yield faster convergence (i.e., fewer iterations are
necessary to obtain $s^{*}$). 

Also it could be noted that instead of the maximal $i$ in
\eqref{eq:12} the minimal or even any other unique choice should in
theory do equally well. In particular, this may give rise to a
different pivoting strategy by shifting the pivoting from the homotopy
algorithm to the labeling function.

\subsection{MATLAB code}
\label{sec:matlab-code}

\begin{lstlisting}[%
label=lst:1, caption={Eaves' algorithm based on the K1 complex
  \cite{Eaves:1972:Homotopies-for-computation-of-fixed-poin:}
  implemented in MATLAB.}, linerange={1-1,36-99999}
%]{/Users/bjoern/Documents/notes/data/76/4416DF-B956-4514-BFEE-7197608E84D0/homotopies/eavesK1.m}
]
  function [kkmpt,noit,succ] = eavesK1(monmap, r, n)
      succ = 0; noit = Inf;  global eavesK1_MAXREFINE
      if isempty(eavesK1_MAXREFINE), eavesK1_MAXREFINE=1e3; end
  %  Initialization of the complex
      q = -eye(n) + diag(ones(n-1,1),-1); v = zeros(n,n+1);
      v(1,1)=1; v(n,1)=1; g=[n,1:n-1];
      for i = 1 : n, v(:,i+1) = v(:,i) + q(:,g(i)); end
  % compute initial labels
      istar=1;  inew=1;  l=zeros(1,n+1);
      for i = 1 : n+1,  
          x = v(:,i)/ norm(v(:,i),1) * r; 
          try
              l(i) = eavesLabel(monmap,x);
          catch exception, 
              error('initial label computation went wrong');
          end
      end
  % the main loop
      lexo=1:n+1;
      for noit=1:eavesK1_MAXREFINE,
  % identify vertex to drop and add new one
          t=find(l==l(inew));
          if t(1) == inew, istar=t(2); else istar=t(1); end
          if istar==1, vnew=v(:,lexo(n+1))+q(:,g(1)); inew=n+1;
              lexo=[lexo(2:n+1) lexo(1)]; v(:,lexo(inew))=vnew;
              g=[g(2:n) g(1)]; l = [l(2:n+1) 0];
          elseif istar < n+1,
              vnew=v(:,lexo(istar-1))+q(:,g(istar));
              inew=istar; lexo=lexo; v(:,lexo(inew))=vnew;
              temp = g(istar); g(istar) = g(istar-1);
              g(istar-1)=temp; clear temp;
              l=[l(1:istar-1) 0 l(istar+1:n+1)];
          else  vnew=v(:,lexo(1))-q(:,g(n)); inew=1;
              lexo=[lexo(n+1) lexo(1:n)]; v(:,lexo(inew))=vnew;
              g=[g(n) g(1:n-1)]; l=[0 l(1:n)];
          end
  % compute label of new vertex:
          kkmpt = vnew/ norm(vnew,1) * r;
          try
              l(inew) = eavesLabel(monmap,kkmpt);
          catch exception, 
              error('successive label computation went wrong')
          end
          if all(feval(monmap,kkmpt)<kkmpt), 
              succ = 1; % success!
              return;
          end
      end % refinements
  end % function eavesK1
  function l = eavesLabel(monmap,x)
      global eavesLabel_distance
      if isempty(eavesLabel_distance), eavesLabel_distance = 1e-2;
      end
      Tx = feval(monmap,x);
      t=find(Tx + eavesLabel_distance<=x); l=max(t);
      if isempty(l), error(['Found point without label. Try ' ...
         'decreasing eavesLabel_distance.']); end
  end  % function eavesLabel
\end{lstlisting}

Of the output arguments \lstinline{kkmpt} denotes the vector
$s^{*}\in\Rnp$, \lstinline{noit} the number of iterations that have
been consumed by the algorithm to find \lstinline{kkmpt}, and
\lstinline{succ} is either $1$ to denote that the algorithm was
successful and \lstinline{kkmpt} is a point of interest and $0$
otherwise.

Of the input arguments \lstinline{monmap} is a function handle to the
monotone map\break $T\colon \Rnp\to\Rnp$ which satisfies $T(0)=0$. The
parameter \lstinline{r} is the radius $r>0$ of the sphere
$S_{r}\subset \Rnp$, where the algorithm tries to find
$s^{*}$. Lastly, \lstinline{n} denotes the dimension of $\Rnp$.

There are two additional global variables that can be tweaked to
modify the performance of the implementation. They are
\lstinline{eavesK1_MAXREFINE} with a default value of 1000 and
\lstinline{eavesLabel_distance} with a default value of
$10^{-2}$. Their meaning is explained below.\medskip

\subsection{Usage}
\label{sec:usage}
To use the above MATLAB function, one has to implement the monotone
map $T$ of interest into a MATLAB function. An example of this is
given in Listing~\ref{lst:2}.
\begin{lstlisting}[caption={A MATLAB implementation of the monotone map defined in Example~\ref{example:sgt-iss}, cf.\ Section~\ref{sec:examples}.}, label=lst:2]
  function y = T(x)
     y=zeros(size(x));
     for k=1:length(x)-1,
       y(k) = y(k) + x(k+1)^(k+1);
       y(k+1) = y(k+1) + x(k)^(1/(k+1));
     end
     y=y/4;
  end
\end{lstlisting}
Now, to compute a point $s^{*}\in\Rnp$ satisfying $Ts^{*}\ll s^{*}$,
$\|s^{*}\|_{1}=r>0$, one calls
\begin{lstlisting}[numbers=none]
  [s_star,noit,succ] = eavesK1(@T,r,n)
\end{lstlisting}
which should yield the desired result. If necessary, the behaviour of
the implementation can be fine-tuned by specifying
\begin{lstlisting}[numbers=none]
  global eavesK1_MAXREFINE eavesLabel_distance
\end{lstlisting}
and assigning new values to these variables for the maximal number of
iterations, and respectively, the parameter $\epsilon$ appearing in
the labeling function~\eqref{eq:12}.

\subsection{Convergence}
\label{sec:covergence}

Assuming arbitrary precision computations and number representation as
well as suitable choices for \lstinline{eavesK1_MAXREFINE} (maximal
number of iterations) and\break \lstinline{eavesLabel_distance} (i.e.,
$\epsilon>0$ in~\eqref{eq:12}), the algorithm does what is expected:

\begin{theorem}
  Let $T\colon\Rnp\to\Rnp$ be monotone and continuous.  Let $r>0$ and
  assume that $Ts\ngeq s$ for all $s\in S_{r}$.  Assume that there
  exists a point $s^{*}\in S_{r}$ such that $Ts^{*}\ll s^{*}$,
  $0<\epsilon< \min \{ s^{*}_{i} - (Ts^{*})_{i}\}$, where $\epsilon$
  is the parameter in~\eqref{eq:12}.  Then the algorithm in
  Listing~\ref{lst:1} produces a point $s^{**}\in S_{r}$ (possibly
  different from $s^{*}$) satisfying $Ts^{**}\ll s^{**}$, provided
  that the maximal number of allowed iterations is large enough.
\end{theorem}
\begin{proof}
  The claim follows from the corresponding more general result
  in~\cite{Eaves:1972:Homotopies-for-computation-of-fixed-poin:}.
\end{proof}

\subsection{Remarks on computational complexity}
\label{sec:remarks-comp-compl}

In each iteration of the algorithm, the map $T\colon \Rnp\to\Rnp$ has
to be evaluated once to compute label of the new vertex. In addition,
$n$ comparisons are necessary to find the old vertex with the same
label as the new one to be dropped.


\section{Algorithmic solution to Problem~\ref{problem:1}}
\label{sec:algor-solut-probl}

Building upon an implementation as in the previous section, Problem~\ref{problem:1} can now be solved as follows:
\begin{enumerate}
\item Given $r>0$ compute $s^{*}$ using the algorithm in the previous
  section, or using a more sophisticated implementation based on one
  of the algorithms proposed in
  e.g. \cite{AllgowerGeorg:1980:Simplicial-and-continuation-methods-for-:,
    AllgowerGeorg:1990:Numerical-continuation-methods:,
    AllgowerGeorg:1997:Numerical-path-following:}.  If this is
  successful then properties 1 and 2 of Problem~\ref{problem:1} are
  already satisfied.
\item Compute $\{T^{k}s^{*}\}_{k\geq0}$. If this is a null-sequence then $T s
  \ngeq s$ for all $s\in [0,s^{*}], s\ne0$, i.e., property 3 of
  Problem~\ref{problem:1} is satisfied.
\end{enumerate}

The assertion of the second step in this short meta-algorithm relies
on the first statement in Theorem~\ref{thm:1}. For if $T^{k}s^{*}$
tends to zero as $k\rightarrow0$ then $s^{*}\in\mathcal{B}$, the region of
attraction. By the ordering of solutions principle also every point
$s\in[0,s^{*}]$ must belong to $\mathcal{B}$ as well. 

The computational complexity of the second step consists of the
computation of only one trajectory of a discrete-time monotone
system. The trajectory has to be bounded, because due to the ordering
of solutions principle it must be confined to the order interval
$[0,s^{*}]$. Furthermore, due to the monotonicity of $T$ it has to be
non-increasing in every component, which allows to terminate further
computation once $\phi_{\eqref{eq:8}}(k,s^{*})$ is sufficiently small
for some $k\geq0$.

\section{Examples}
\label{sec:examples}

In this section we consider a few numerical examples. The first is a
nonlinear map $T\colon\Rnp\to\Rnp$ which can be defined for any
$n\geq2$. For this map $T$ it is known that system~\eqref{eq:9} is
input-to-state-stable, implying that the origin is globally
asymptotically stable for system~\eqref{eq:8}, and hence the Eaves K1
algorithm should produce an $s^{*}$ for arbitrary $n\geq2$ and $r>0$.

The second example is a statistic generated from randomly chosen
nonnegative matrices $A\in\Rpnn$ with spectral radius less than one. 

A third example shows that for a given $x(0)$ the pure iteration of
$x(k+1)=Tx(k)$ does in general not produce an $x(k)$ such that
$Tx(k)\ll x(k)$, even if $k\geq0$ is large.

\begin{example}
  \label{example:sgt-iss}
  Let $n\geq2$.  Consider the nonlinear map $T\colon\Rnp\to\Rnp$
  defined in
  \cite[Example~IV.1]{Ruffer:2009:Small-gain-conditions-and-the-comparison:}
  given by
  \[ 
  \big(Ts)_{i} = \frac{1}{4}\big(s_{i-1}^{1/i} +
  s_{i+1}^{i+1}\big), \quad s\in\Rnp,
  \] 
  with the convention that $s_{0}=s_{n+1}=0$.  For example, in the
  case $n=5$ one has
  \[
  Ts = \frac{1}{4}
  \begin{pmatrix}
    s_{2}^{2}\\
    \sqrt{s_{1}} + s_{3}^{3}\\
    \sqrt[3]{s_{2}} + s_{4}^{4}\\
    \sqrt[4]{s_{3}} + s_{5}^{5}\\
    \sqrt[5]{s_{4}}\\
  \end{pmatrix}\,.
  \]
  Observe that $T(0)=0$, and that $T$ is obviously monotone and
  continuous.  It has been shown in
  \cite[Example~IV.1]{Ruffer:2009:Small-gain-conditions-and-the-comparison:}
  that the induced system~\eqref{eq:9} is input-to-state
  stable. Moreover, it can be verified directly that for every $r>0$,
  the vector
  \[
  p(r) =
  \begin{pmatrix}
    {r}\\
    \sqrt{r}\\
    \sqrt[3!]{r}\\
    \vdots\\
    \sqrt[n!]{r}
 \end{pmatrix}
 \] 
 satisfies $Tp(r)\ll p(r)$. Since $p\colon\Rp\to\Rnp$ is continuous,
 and in every component unbounded, for every $r>0$ there must exists a
 point $s^{*}\in S_{r}$ such that $Ts^{*}\ll
 s^{*}$. Figure~\ref{fig:sgt-iss} shows how long it takes (in terms of
 iterations) for the Eaves' algorithm to find such a point $s^{*}$.
 
 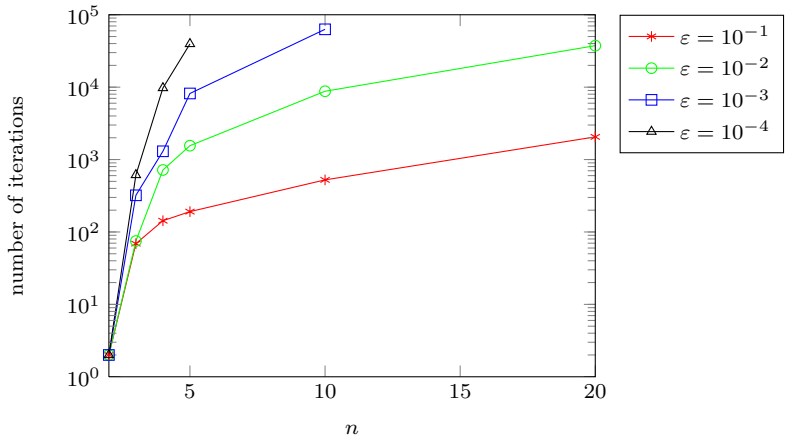
\begin{figure}[htbp]
   \centering
   \beginpgfgraphicnamed{Fig1}
   \begin{tikzpicture}
     \begin{semilogyaxis}[
       axis on top,
       scale only axis,
       width=6.4cm,
       height=4.8cm,
       xmin=2, xmax=20,
       ymin=1, ymax=100000,
       legend entries={$\epsilon=10^{-1}$,$\epsilon=10^{-2}$,$\epsilon=10^{-3}$,
         $\epsilon=10^{-4}$},
       xlabel={$n$},
       ylabel={number of iterations}
       ]

       \addplot [
       color=red,
       solid,
       mark=asterisk,
       mark options={solid}
       ]coordinates{
         (2,2) (3,69) (4,143) (5,191) (10,524) (20,2057)
       };

       \addplot [
       color=green,
       solid,
       mark=o,
       mark options={solid}
       ]coordinates{
         (2,2) (3,75) (4,718) (5,1554) (10,8779) (20,37418)
       };

       \addplot [
       color=blue,
       solid,
       mark=square,
       mark options={solid}
       ]coordinates{
         (2,2) (3,321) (4,1301) (5,8181) (10,62939) (20,0)
       };

       \addplot [
       color=black,
       solid,
       mark=triangle,
       mark options={solid}
       ]coordinates{
         (2,2) (3,613) (4,9738) (5,39548) (10,0)
       };
    \end{semilogyaxis}
  \end{tikzpicture}
  \endpgfgraphicnamed
  \caption{Number of iterations needed by Eaves' K1 algorithm for
    $r=10$, and different choices of $n$ and $\epsilon>0$ in
    Example~\ref{example:sgt-iss}. The maximal number of iterations
    was set to $10^{5}$. It took 67 seconds to generate this whole
    plot on an 2.4 GHz Intel MacBook.}
   \label{fig:sgt-iss}
 \end{figure}
\end{example}

\begin{example}
  \label{example:linearstats}
  It is relatively easy to generate many positive $n\times n$ matrices
  with a specified spectral radius in MATLAB. If the spectral radius
  $\rho(A)$ of a nonnegative matrix $A$ is less than one then it
  defines a monotone mapping that should allow for a point $s^{*}\in
  S_{r}$ with $As^{*}\ll s^{*}$. Here we have chosen the spectral
  radius to be $\rho(A)=0.8$ and have generated a number of matrices
  $A\in\Rpnn$ for different choices of $n$. Again we have applied the
  algorithm in Listing~\ref{lst:1} and counted the number of
  iterations needed. The outcome is plotted in
  Figure~\ref{fig:linear}.

  \begin{figure}[htbp]
    \centering
    \beginpgfgraphicnamed{Fig2}
    \begin{tikzpicture}
      \begin{semilogyaxis}[
        axis on top,
        scale only axis,
        width=6.4cm,
        height=4.8cm,
        xmin=2, xmax=20,
        ymin=1, ymax=100000,
        legend entries={$\epsilon=10^{-1}$,,$\epsilon=10^{-2}$,,$\epsilon=10^{-3}$,,
          $\epsilon=10^{-4}$,},
        xlabel={$n$},
        ylabel={number of iterations}
        ]

        \addplot[ color=red, solid, mark=asterisk, mark options={solid}
        ] plot[error bars/.cd, y dir=plus, y explicit
        ]
        table[x=n,y=m,y error=u] {example2_1.dat};
        \addplot[ color=red, solid, mark=asterisk, mark options={solid}
        ] plot[error bars/.cd, y dir=minus, y explicit
        ]
        table[x=n,y=m,y error=l] {example2_1.dat};

        \addplot[ color=green, solid, mark=o, mark options={solid}
        ] plot[error bars/.cd, y dir=plus, y explicit
        ]
        table[x=n,y=m,y error=u] {example2_2.dat};
        \addplot[ color=green, solid, mark=o, mark options={solid}
        ] plot[error bars/.cd, y dir=minus, y explicit
        ]
        table[x=n,y=m,y error=l] {example2_2.dat};

        \addplot[ color=blue, solid, mark=square, mark options={solid}
        ] plot[error bars/.cd, y dir=plus, y explicit
        ]
        table[x=n,y=m,y error=u] {example2_3.dat};
        \addplot[ color=blue, solid, mark=square, mark options={solid}
        ] plot[error bars/.cd, y dir=minus, y explicit
        ]
        table[x=n,y=m,y error=l] {example2_3.dat};
    \end{semilogyaxis}
    \end{tikzpicture}
    \endpgfgraphicnamed
    \caption{Mean number of iterations needed for linear maps with
      spectral radius $\rho(A)=0.8$ together with minimal and maximal
      number of iterations needed. Each data point corresponds to
      10 
      randomly chosen matrices (using MATLAB's
      \lstinline{rand(n)}). Generating this plot took 205 seconds (not
      taking into account generation of the test matrices). The
      maximal number of iterations was set to $10^{5}$.%
    }
    \label{fig:linear}
  \end{figure}
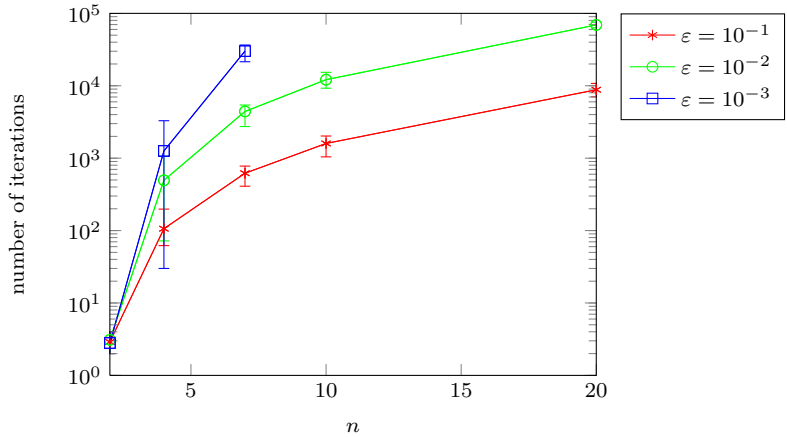
\end{example}

\begin{example}
  \label{example:flipflop}
  Consider the monotone map $T\colon\R^{2}_{+}\longrightarrow\R^{2}_{+}$ given by
  \break
  $T(x)= \big(\sqrt{x_{2}}, \lambda x_{1}^{2}\big)^{T}$, with
  $\lambda\in(0,1)$. Obviously $T(0)=0$ and it is easy to check that
  for any $x\in\Rnp$, $T^{k}x\to0$ as $k\to\infty$. Hence, $T$ is a
  contraction. Yet, given $x\in\R^{2}_{+}$, if not already $Tx\ll x$ then there exists no
  $k\geq 1$ such that $T^{k+1}x\ll T^{k}x$. For we have,
  \[
  T^{2}x =
  \begin{pmatrix}
    \sqrt{\lambda} x_{1}\\
    \lambda x_{2}
  \end{pmatrix},
  \quad
  Tx =
  \begin{pmatrix}
    \sqrt{x_{2}}\\
    \lambda x_{1}^{2}
  \end{pmatrix},
  \quad
  x=
  \begin{pmatrix}
    x_{1}\\
    x_{2}
  \end{pmatrix}.
  \]
  Assuming that $Tx\not\ll x$, we have either
  $(Tx)_{1}=\sqrt{x_{2}}\geq x_{1}$, which implies $(T^{2}x)_{2}x=
  \lambda x_{2}\geq\lambda x_{1}^{2}=(Tx)_{2}$. Otherwise, we have
  $(Tx)_{2}=\lambda x_{1}^{2}\geq x_{2}$, implying that
  $(T^{2}x)_{1}=\sqrt{\lambda}x_{1}\geq \sqrt{x_{2}}=(Tx)_{1}$. This
  deduction repeats inductively.  As a consequence, a pure iteration
  of the map $T$ cannot yield a solution to Problem~\ref{problem:1}.
\end{example}

\section{Conclusions}
\label{sec:conclusions}

\begin{acknowledgements}
  B.~S.~R\"uffer has been partially supported by the Australian
  Research Council's \emph{Discovery Projects} funding scheme (project
  number DP0880494) 
  and the Japan Society for the Promotion of Science. F.~R.~Wirth has
  been supported by the German Science Foundation (DFG) within the
  priority programme 1305: Control Theory of Digitally Networked
  Dynamical Systems.

  The authors would like to thank Priv.-Doz.\ Dr.\ Sergey Dashkovskiy
  for numerous valuable discussions prior to this work.
\end{acknowledgements}

\def\cprime{$'$}

\end{document}